\documentclass[reqno,12pt]{amsart}\usepackage[sort&compress,numbers]{natbib}
\usepackage{amsfonts}
\usepackage{mathdots,amssymb,graphicx,amssymb,amsmath,amsopn, mathtools,mathabx}
\usepackage[utf8]{inputenc}
\usepackage{enumitem}
\usepackage{tikz-cd}
\usepackage{scalerel}
\usepackage{pgfplots}
\usepackage{amscd}
\usepackage{latexsym}
\usepackage{xcolor}
\usepackage{amscd}
\usepackage{pb-diagram}
\usepackage{mathrsfs}
\usepackage{accents}
\usepackage[textwidth=15cm,textheight=22.5cm,headheight=0.6cm,headsep=1cm,centering]{geometry}
\usepackage{fancyhdr}

\usepackage{float}
\usepackage{array}
\usepackage{graphicx}
\usepackage{subcaption}
\definecolor{brightmaroon}{rgb}{0.76, 0.13, 0.28}
\definecolor{airforceblue}{rgb}{0, 0.25, 0.77}
\definecolor{myOrange}{rgb}{1,0.5,0}
\definecolor{brightmaroon}{rgb}{0.76, 0.13, 0.28}
\definecolor{airforceblue}{rgb}{0, 0.4, 0.66}

\allowdisplaybreaks
\pgfplotsset{compat=1.18}

\theoremstyle{plain}

\newtheorem{teo}{Theorem}[section]

\newtheorem{pro}[teo]{Proposition}
\newtheorem{coro}[teo]{Corollary}
\newtheorem{defi}[teo]{Definition}
\newtheorem{remark}[teo]{Remark}

\numberwithin{equation}{section}
\newcommand{\An}{\mathcal{A}}
\newcommand{\Bn}{\mathcal{B}}

\newcommand{\un}{\mathbf{u}}

\newcommand{\Su}{\mathbf{S}}

\newcommand{\Pn}{\mathbf{P}}

\numberwithin{equation}{section}
\newcommand{\prodint}[1]{\left\langle{#1}\right\rangle}

\usepackage{color}

\newcommand*\pFq[2]{{}_{#1}F_{#2}}
\begin{document}
\title[A class of Truncated Freud polynomials]{A class of Truncated Freud polynomials}
	
\author[J. C García-Ardila, F.  Marcellán, M. E. Marriaga]{Juan C. García-Ardila,  Francisco Marcellán, Misael E. Marriaga}
\address[J. C. Garc\'ia-Ardila]{Departamento de Matem\'atica Aplicada a la Ingenier\'ia Industrial \\Universidad Polit\'ecnica de Madrid\\ Calle Jos\'e Gutierrez Abascal 2, 28006 Madrid, Spain.}\email{juancarlos.garciaa@upm.es}
\address[F. Marcellán]{Departamento de Matemáticas, Universidad Carlos III de Madrid, Leganés, Spain} \email{pacomarc@ing.uc3m.es}

\address[M. E. Marriaga]{Departamento de Matem\'atica Aplicada, Ciencia e Ingeniería de Materiales y Tecnología Electrónica, Universidad Rey Juan Carlos, Mo\'stoles Spain} \email{misael.marriaga@urjc.es}

\begin{abstract}
Consider the following truncated Freud  linear functional $\un_z$ depending on a parameter $z$,
$$\prodint{\un_z,p}=\int_0^\infty p(x)e^{-zx^4}dx,\quad z>0.$$
The aim of this work is to analyze the properties of the sequence of orthogonal polynomials $(P_n)_{n\geq 0}$ with respect to $\un_z$.  Such a linear functional is semiclassical and, as a consequence,  we get the system of nonlinear difference equations (Laguerre-Freud equations) that the coefficients of the three-term recurrence satisfy. The asymptotic behavior of such coefficients is given. On the other hand, the raising and lowering operators associated with such a linear functional are obtained, and thus a second-order linear differential equation of holonomic type that $(P_n)_{n\geq 0}$ satisfies is deduced. From this fact, an electrostatic interpretation of their zeros is given. Finally, some illustrative numerical tests concerning the behavior of the least and greatest zeros of these polynomials are presented.
\end{abstract}
\maketitle	


\textbf{Corresponding author:} Misael E. Marriaga, misael.marriaga@urjc.es

 \section{Introduction}
Consider the sequence of orthogonal polynomials associated with a linear functional $\un$ defined from the weight function $w(x)= e^{-x^2}$ supported on the positive real semi-axis by 
$$
\langle \un, p\rangle=\int_{0}^{\infty} p(x) e^{-x^2}dx.
$$
This linear functional belongs to a wide class of linear functionals which is known in the literature as semiclassical \cite{Ma91}. Indeed, they are semiclassical of class $s=1$, see \cite{Angel}. The concept of class allows us to introduce a hierarchy of linear functionals that constitutes an alternative way to the Askey tableau based on the hypergeometric character of the corresponding sequence of orthogonal polynomials. Semiclassical orthogonal polynomials appear in the seminal paper \cite{shohat}, where weight functions whose logarithmic derivatives are rational functions were considered. A second-order linear differential equation satisfied by the corresponding sequence of orthogonal polynomials is obtained therein. The theory of semiclassical orthogonal polynomials has been developed by P. Maroni (see \cite{Ma87,Ma91}) who combined techniques of functional analysis, distribution theory, Fourier and $z$-transforms, ordinary linear differential equations, among others, in order to deduce algebraic and structural properties of such polynomials. Many of the most popular families of orthogonal polynomials coming from \textit{Mathematical Physics} as Hermite, Laguerre, Jacobi, and Bessel are semiclassical of class $s=0$, but other families appearing in the literature are also semiclassical. For instance, semiclassical  families of class $s=1$ are described in \cite{Belmehdi}. As an example, in \cite{DM22} truncated Hermite polynomials, associated with the normal distribution supported on a symmetric interval of the real line, have been considered. They are semiclassical of class $s=2.$\\

The coefficients of the three-term relation satisfied by orthogonal polynomials associated with a semiclassical linear functional satisfy coupled non linear difference equations (Laguerre-Freud equations, see \cite{Said}) which are related to discrete Painlev\'e equations, see \cite{Magnus1, Magnus2, Walter1}, among others. On the other hand, ladder operators (lowering and raising) are associated with semiclassical orthogonal polynomials. Their construction is based on the so-called first structure relation (\cite{Ma91}) they satisfy. As a consequence, one can deduce a second-order linear differential equation that the polynomials satisfy. Notice that the coefficients are polynomials with degrees depending on the class.\\

In the framework of random matrices, several authors have considered Gaussian unitary ensembles with one or two discontinuities (see \cite{Lyu1}). Therein it is proved that the logarithmic derivative of the Hankel determinants generated by the normal (Gaussian) weight with two jump discontinuities in the scaled variables tends to the Hamiltonian of a coupled Painlev\'e II system and it satisfies a second-order PDE. The asymptotics of the coefficients of the three-term recurrence relation for the corresponding orthogonal polynomials are deduced. They are related to the solutions of the coupled Painlev\'e  II system. The techniques are based on the analysis of ladder operators associated with such orthogonal polynomials. For more information, see \cite{Chen97,Chen05}.\\

Integrable systems also provide nice examples of semiclassical orthogonal polynomials. They are related to time perturbations of a weight function (Toda, Langmuir lattices). For more information, see \cite{Walter2}.\\

In the literature, see \cite{Shizgal1981}, among others, Gaussian quadrature rules are analyzed for weight functions $w(x)= x^{p} e^{-x^2}, p>0, $ supported on the positive real semi-axis. The corresponding orthogonal polynomials are called Maxwell polynomials. These polynomials constitute a useful tool for the solution of the Boltzman and/or Fokker-Planck equations. It is important to point out that the coefficients of the three-term recurrence relations  satisfy non-linear equations, which are numerically unstable. In \cite{Shizgal1981}, extended precision arithmetic is used to generate the recurrence coefficients to high order. As a consequence, the polynomials and associated quadrature weights and nodes are deduced.\\

Maxwell polynomials are used in various fields for different purposes, such as in pseudo-spectral collocation schemes. Among the numerous characteristics that make these polynomials desirable for pseudo-spectral discretization schemes for the velocity variable, one can emphasize their capability to handle semi-infinite intervals, the convenient distribution of their zeros that balances clustering around zero, where increased resolution is often needed, with the sampling at increasingly larger distances from the origin, and their optimal location for the computation of integrals involving a Maxwell-Boltzmann distribution. For further details, see \cite{landreman,SANCHEZVIZUET2018,Shizgal2015}.\\

The aim of our contribution is to analyze the sequence of monic orthogonal polynomials associated with a linear functional $\un_z$ defined from a weight function $w(x)= e^{-zx^4}$ supported on the positive real semi-axis by 
$$
\langle \un_z, p\rangle=\int_{0}^{\infty} p(x) e^{-zx^4}dx, z>0.
$$
Such a kind of weight function appears in the framework of high electric field strength, where the steady state electron velocity distribution of electrons dilutely dispersed in a background gas of atoms is given by the so-called  Druyvesteyn distribution function, assuming that the electron-atom momentum transfer cross section is constant~\cite{Huxley}. This distribution (weight) function has the form $w(x)=x^2 \exp (-ax^4),$ where $x$ denotes the reduced speed of electrons. In \cite{Clarke} the authors analyze the coefficients of the three-term recurrence that the sequence of orthogonal polynomials associated with the weight $w(x)=x^p \exp (-ax^4)$, $p=0, 1, 2, \ldots,$ supported in the positive real semi-axis satisfies. In particular, a first approach to their asymptotic expansions is given, as well as some open problems related to Gaussian quadrature formulas, when $p=2$ are stated.\\

The structure of the manuscript is as follows. In Section \ref{background}, a basic background on semiclassical linear functionals is provided. Section \ref{section3} is focused on the study of truncated Freud linear functionals, that is, the case $\langle \un_z, p\rangle=\int_{0}^{\infty} p(x) e^{-zx^4}dx,$ $z\in(0,\infty)$, a semiclassical linear functional of class $3$. Furthermore, in Section \ref{Section4}, a system of nonlinear difference equations that the corresponding sequence of monic orthogonal polynomials satisfies is deduced. As a consequence, the asymptotic behavior of such coefficients in terms of $n$ is deduced. In Section \ref{Sec, Holo}, we study the ladder operators associated with such a sequence of monic orthogonal polynomials. Thus, we get a second-order linear differential equation with polynomial coefficients for such orthogonal polynomials. Section \ref{section6} focuses on the study of the parameters in the three-term recurrence relation in terms of $z.$ Finally, in Section \ref{elect_Section} we provide an electrostatic interpretation of the distribution of their zeros, as well as some numerical tests about the behavior of the least and greatest zeros of these polynomials.

\section{Basic background}\label{background}
Let $\mathbb{P}$ denote the linear space of polynomials with complex coefficients. Given a linear functional $\mathbf{u}$ that acts on $\mathbb{P}$, the $n$th moment of $\un$ is defined as $\mu
_n:=\prodint{\un,x^n}$, $n\geq0$. For $q\in\mathbb{P}$ we define the linear functional $q(x)\un$ as
\begin{equation*}
\prodint{q\un,p}=\prodint{\un,q p}, \quad p\in\mathbb{P}.
\end{equation*}
The distributional derivative of $\un$ is the linear functional $D\un$ such that
		$$
            \prodint{D\un,p}=-\prodint{\un,p^\prime}, \quad p\in\mathbb{P}.
            $$
The linear functional $\un$ is said to be quasi-definite if  there exists a sequence of monic polynomials $(P_n)_{n\geq0}$ such that $\deg{P_n(x)}=n$ and $\prodint{\un,P_nP_m}=K_{n}\delta_{n,m},$ where $\delta_{n,m}$ is the Kronecker symbol and $K_{n}\ne 0$ for $n\geqslant 0$ (see \cite{Chi,GMM21}). If $\un$ is quasi-definite, then the sequence $(P_n)_{n\geq0}$ is said to be the \textit{sequence of monic orthogonal polynomials} (SMOP) associated with $\un$.  If $K_{n}>0,$ then the linear functional $\un$ is said to be positive definite. \\
	
Suppose that there exist two sequences of complex numbers $( a_n)_{n\geq1}$ and $(b_n)_{n\geq0}$, with $ a_n\ne 0$,  and let $(P_n)_{n\geq0}$ be a sequence of monic polynomials generated by a three-term recurrence relation (TTRR) 	
\begin{equation}\label{ttrrr}
\begin{aligned}
			x\,P_{n}(x)&=P_{n+1}(x)+b_n\,P_{n}(x)+ a_{n}\,P_{n-1}(x),\quad n\geq0,\\
			P_{-1}(x)&=0, \ \ \ \ \ P_{0}(x)=1.
		\end{aligned}
	\end{equation}
	Then,  by Favard's Theorem, the above is equivalent to the existence of a unique  quasi-definite  linear functional $\un$ such that $(P_n)_{n\geq0}$ is its corresponding SMOP, see \cite{Chi}.\\

Sometimes, it is more convenient to work with the matrix expression of the three-term recurrence relation \eqref{ttrrr}, that is, if $\Pn=(P_0,P_1,\cdots)^T$, then the recurrence relation can be written in matrix form
\begin{equation}\label{jacobi}
x\Pn=\mathbf{J}\Pn \quad \text{with}\quad \mathbf{J}= \begin{pmatrix}
		b_0&1&0\\
		a_1&b_1&1&0	\\
		0&a_2&b_2&1&0\\
		&0&a_3&b_3&1&0\\
		&&\ddots&\ddots&\ddots&\ddots&\ddots\\
	\end{pmatrix}.
\end{equation}
The matrix $\mathbf{J}$ is known in the literature as \textit{monic Jacobi matrix}.\\

If $(P_n)_{n\geq 0}$ is a sequence of monic orthogonal polynomials with respect to a linear functional $\un$, then using the \textit{Christoffel-Darboux formula} (see \cite{GMM21}) we can get the  following confluent formula
\begin{equation*}
    \sum_{k=0}^n\frac{P^2_{k}(x)}{h_k}=\frac{P^{\prime}_{n+1}(x)P_{n}(x)-P^\prime_{n}(x)P_{n+1}(x)}{h_n},
\end{equation*}
where $h_k=\prodint{\un,P_k^2}.$
	\begin{defi}
		Given a quasi-definite linear functional $\un$ with moments $(\mu_n)_{n\geq0}$, the formal series
		\begin{equation*}
			\Su_{\un}(t):=\sum_{n=0}^\infty \dfrac{\mu_n}{t^{n+1}}
		\end{equation*}
		is called the \textit{Stieltjes function} associated with $\un$.
	\end{defi}

	
Next, we present some important results concerning semiclassical linear functionals.
	\begin{defi}[\cite{Ma87}]
		A quasi-definite functional $\mathbf{u}$ is said to be semiclassical if there exist non-zero polynomials $\phi(x)$ and $\psi(x)$ with $\deg\phi(x)=:r\ge 0$ and $\deg\psi(x)=:t\ge 1$, such that $\mathbf{u}$ satisfies the distributional Pearson equation
		\begin{equation}\label{pearson-semic}
			D(\phi\,\mathbf{u})+\psi\,\mathbf{u}=\mathbf{0}.
            \end{equation}
		An SMOP associated with a semiclassical linear functional $\mathbf{u}$ is called  semiclassical.
	\end{defi}
	Because a semiclassical linear functional satisfies many Pearson equations, we give the following definition.
	\begin{defi}
		The class of a semiclassical functional $\mathbf{u}$ is defined as
		\begin{equation*}
			\mathfrak{s}(\mathbf{u}):= \min \max\{\deg \phi(x)-2, \deg\psi(x)-1 \},
		\end{equation*}
		where the minimum is taken over all pairs of polynomials $\phi(x)$ and $\psi(x)$ so that \eqref{pearson-semic} holds.
	\end{defi}
	\begin{pro}[\cite{GMM21, Ma91}]\label{sim_cond}
		Let $\mathbf{u}$ be a semiclassical linear functional, and let $\phi(x)$ and $\psi(x)$ be non-zero polynomials with $\deg\phi(x)=:r$ and $\deg \psi(x)=:t$, such that \eqref{pearson-semic} is satisfied. Let $s := \max(r-2,t-1)$. Then $s = \mathfrak{s}(\mathbf{u})$  if and only if
		\begin{equation}\label{prodformula}\prod_{c:\,\phi(c)=0}\left(|\psi(c)+\phi^\prime(c)|+|\langle\mathbf{u},\theta_c\psi+\theta^2_c\phi\rangle|\right)>0.
		\end{equation}
		Here, $(\theta_c f ) (x)=\dfrac{f(x)-f(c)}{x-c}, x\neq c, (\theta_c f ) (c)= f ' (c).$
	\end{pro}
	
According to the class, there is a hierarchy of semiclassical linear functionals. The class $s=0$ is constituted by the classical linear functionals (Hermite, Laguerre, Jacobi and Bessel). The class $s=1$ has been studied in \cite{Belmehdi}.

\section{Truncated Freud polynomials}\label{section3}
In this section we study the Freud linear functional $\un_z$ depending on a parameter~$z$
\begin{equation}\label{lfpd}
\prodint{\un_z,p}=\int_0^\infty p(x)e^{-zx^4}dx,\quad z>0,
\end{equation}

as well as its corresponding sequence of monic orthogonal polynomials $(P_n)_{n\geq 0}.$ We will prove that $\un_z$ is semiclassical   and leverage this property to derive characteristics of both the moments of $\un_z$  and the parameters of the TTRR satisfied by $(P_n)_{n\geq 0}$.\\

Let $(\mu_n(z))_{n\geq 0}$ be the sequence of moments of $\un_z$ and let $n=4m+k$ with $m=0,1\ldots,$ and $k=0,1,2,3$.  By setting $t=zx^4$ we
can write
\begin{equation*}
	\begin{aligned}
\mu_{4m+k}(z)=\prodint{\un_z,x^{4m+k}}&=\int_0^\infty x^{4m+k}e^{-zx^4}dx
		=\dfrac{z^{-m-(k+1)/4}}{4}\,\int_0^\infty t^{m+\frac{k-3}{4}} e^{-t}dt.	
	\end{aligned}
\end{equation*}
Observe that the last integral is $\Gamma(m+\frac{k+1}{4})$, where $\Gamma(x)$ is the Gamma function.
With this in mind, we have that for $m\geqslant 0$ and $k=0,1,2,3,$
\begin{equation*}
	\mu_{4m+k}(z)=\dfrac{z^{-m-(k+1)/4}}{4}\,\Gamma\left(m+\dfrac{k+1}{4}\right),
    \end{equation*}
or,
\begin{equation}\label{moments}
	\mu_{n}(z)=\dfrac{z^{-(n+1)/4}}{4}\,\Gamma\left(\dfrac{n+1}{4}\right),  \quad n\geqslant 0.
    \end{equation}
\begin{teo}
The functional $\un_z$ is semiclassical of class $3$. Moreover, $\un$ satisfies the Pearson equation \begin{equation}\label{pearson_eq}
	D(\phi\,\un_z)+\psi\,\un_z=\mathbf{0},
\end{equation}
with $\phi(x)=x$ and $\psi(x)=4zx^4-1.$
\end{teo}
\begin{proof} Using integration by parts, for all $p\in \mathbb{P}$,
\begin{align*}
\prodint{D(\phi\,\un_z),p}&=-\prodint{\un_z,\phi\, p^\prime}=-\int_{0}^\infty x\, p^\prime(x) e^{-zx^4}dx\\
	&=\left.-p(x)xe^{-zx^4}\right|_{0}^{\infty}+\int_{0}^\infty p(x)\left(1-4zx^4\right) e^{-zx^4}dx \\
	&=-\prodint{\psi\,\un_z,p\,}.
\end{align*}
Next, we determine the class. In this case, \eqref{prodformula} is read as $4z|\langle \un_z, x^3\rangle |$. Taking into account \eqref{moments}, we get
$$
4z|\langle \un_z, x^3\rangle |=1 > 0.
$$
Then, by Proposition \ref{sim_cond}, the class of $\un_z$ is $\mathfrak{s}(\mathbf{u}_z)=3$.
\end{proof}
Applying \eqref{pearson_eq}  to the monomial $p(x)=x^n$ we get the following recurrence relation for the moments.
\begin{coro}
The moments $(\mu_n(z))_{n\geq0}$ with $\mu_n(z):=\langle \un_z,x^n\rangle$ satisfy the following fourth-order recurrence equation,
\begin{equation}\label{recmoments}
4z\mu_{n+4}(z)-(n+1)\mu_n(z)=0,\qquad n\geq 0,
\end{equation}
with $\mu_0(z),\mu_1(z),\mu_2(z),\mu_3(z)$ given in \eqref{moments}.
\end{coro}

The recurrence relation \eqref{recmoments} allows us to obtain a non-homogeneous first order linear ODE satisfied by the Stieltjes function associated with $\un_z$.

\begin{pro}
The Stieltjes function $\Su_{\un_z}(t;z)$ associated with the linear functional $\un_z$ satisfies a non-homogeneous first order linear ODE. Namely,
$$t\partial_t\Su_{\un_z}(t;z)+4zt^4\Su_{\un_z}(t;z)=4z\left(\mu_3(z)+t\mu_2(z)+t^2\mu_1(z)+\mu_0(z)t^3\right).$$
\end{pro}
\begin{proof}
Note that  by \eqref{recmoments}, we have
\begin{align*}
	\Su_{\un_z}(t;z)&=\dfrac{\mu_0(z)}{t}+\dfrac{\mu_1(z)}{t^2}+\dfrac{\mu_2(z)}{t^3}+\dfrac{\mu_3(z)}{t^4}+\dfrac{1}{4z}\sum_{n=0}^\infty4z\,\dfrac{\mu_{n+4}(z)}{t^{n+5}}\\
    &=\dfrac{\mu_0(z)}{t}+\dfrac{\mu_1(z)}{t^2}+\dfrac{\mu_2(z)}{t^3}+\dfrac{\mu_3(z)}{t^4}+\dfrac{1}{4z}\sum_{n=0}^\infty(n+1)\dfrac{\mu_{n}(z)}{t^{n+5}}.
\end{align*}
Since
$$\partial_t\Su_{\un_z}(t;z)=-\sum_{n=0}^\infty(n+1)\dfrac{\mu_{n}(z)}{t^{n+2}},$$
it follows that
$$\Su_{\un_z}(t;z)
    =\dfrac{\mu_0(z)}{t}+\dfrac{\mu_1(z)}{t^2}+\dfrac{\mu_2(z)}{t^3}+\dfrac{\mu_3(z)}{t^4}-\dfrac{1}{4z\,t^3}\partial_t\Su_{\un_z}(t;z).$$
The desired result is obtained by multiplying both sides of the above equation by~$4zt^4.$
\end{proof}
\section{Laguerre-Freud equations}\label{Section4}
If $\un$ is a semiclassical linear functional satisfying the Pearson equation \eqref{pearson_eq}, then the coefficients of the corresponding TTRR \eqref{ttrrr} satisfied by its SMOP $(P_n)_{n\geq 0}$ satisfy a nonlinear system of equations derived from the so-called \textit{Laguerre-Freud equations} \cite{Said}.
\begin{align}
\label{1eq}\prodint{\psi\un,P_n^2}&=-\prodint{D(\phi\un),P_n^2},\\
\label{2eq}\prodint{\psi\un,P_{n+1}P_n}&=-\prodint{D(\phi\un_z),P_{n+1}P_n}.
\end{align}
The iteration  of \eqref{ttrrr} yields
\begin{multline*}
    x^2P_{n} (x)=P_{n+2}(x)+(b_{n+1}+b_n)P_{n+1}(x)+(a_{n+1}+b_n^2+a_n)P_{n}(x)\\+a_n(b_n+b_{n-1})P_{n-1}(x)+a_na_{n-1}P_{n-2}(x).
\end{multline*}
If we define $R_n=a_{n+1} +b_n^2+a_n$ and $T_n=a_n(b_n+b_{n-1})$, then 
\begin{equation}\label{ttrr2}x^2P_{n}(x)=P_{n+2}(x)+(b_{n+1}+b_{n})P_{n+1}(x)+R_nP_{n}(x)+T_nP_{n-1}(x)+a_na_{n-1}P_{n-2}(x).
\end{equation}
Iterating the above formula, we get
\begin{equation}\label{4rr}
x^4P_{n}(x)=P_{n+4}(x)+\sum_{k=n-4}^{n+3}\beta_{n,k}(z)P_{k}(x),\end{equation}
where
\begin{equation*}
\begin{aligned}
 \beta_{n,n+3}(z)&=\sum_{i=0}^3b_{n+i},\\
 \beta_{n,n+2}(z)&=R_{n+2}+(b_{n+1}+b_n)(b_{n+2}+b_{n+1})+R_n,\\
 \beta_{n,n+1}(z)&=T_{n+2}+(b_{n+1}+b_n)(R_{n+1}+R_n)+T_n,\\
 \beta_{n,n}(z)&=a_{n+2}a_{n+1}+(b_{n+1}+b_n)T_{n+1}+R_n^2+(b_{n-1}+b_n)T_{n}+a_na_{n-1},\\
  \beta_{n,n-1}(z)&=(b_{n+1}+b_{n})a_{n+1}a_{n}+(R_{n}+R_{n+1})T_n+(b_{n-1}+b_{n-2})a_{n}a_{n-1},\\
  \beta_{n,n-2}(z)&=a_na_{n-1}(R_n+R_{n+2})+T_nT_{n-1},\\
  \beta_{n,n-3}(z)&=a_{n-1}a_{n-2}T_{n}+a_{n}a_{n-1}T_{n-2},\\
  \beta_{n,n-4}(z)&=\prod_{i=n-3}^na_{i}.
\end{aligned}
\end{equation*}
Moreover, if we define $\mathbf{B}=(\beta_{n,k})_{n,k=0}$, where $\beta_{n,k}=0$ for all $k\ne n-3,\ldots, n+3$, then $$x^4\mathbf{P}=\mathbf{B}\mathbf{P}.$$
This implies that $\mathbf{B}=\mathbf{J}^4,$ where $\mathbf{J}$ is the monic Jacobi matrix given in \eqref{jacobi}.

If we write
	\begin{equation}\label{rev}
		P_n(x)=x^n+\sum_{k=0}^{n-1}\lambda_{n,k}x^k,
	\end{equation}
	then substituting \eqref{rev} into \eqref{ttrrr} and comparing the coefficients, we get
	\begin{equation}\label{cases}
		\begin{cases}
			b_n=\lambda_{n,n-1}-\lambda_{n+1,n},  \\
			\lambda_{n,n-1}b_n+a_n=\lambda_{n,n-2}-\lambda_{n+1,n-1}.
		\end{cases}
	\end{equation}
	Finally, differentiating \eqref{rev} and considering \eqref{cases}, we obtain
    \begin{align}
\notag P'_n(x)&=nx^{n-1}+(n-1)\lambda_{n,n-1}x^{n-2}+\mathcal{O}(x^{n-3})\\
		\notag &=nP_{n-1}(x)+((n-1)\lambda_{n,n-1}-n\lambda_{n-1,n-2})P_{n-2}(x)+\mathcal{O}(x^{n-3})\\
\label{der} &=nP_{n-1}(x)-(nb_{n-1}+\lambda_{n,n-1})P_{n-2}(x)+\mathcal{O}(x^{n-3}).
    \end{align}

\begin{teo}
    The coefficients of the TTRR  \eqref{ttrrr} associated with the linear functional $\un_z$ satisfy the Laguerre-Freud equations
    \begin{align}
        &4z\left[a_{n+2}a_{n+1}+T_{n+1}(b_{n+1}+b_n)+R^2_n+T_n(b_n+b_{n-1})+a_na_{n-1}\right]=2n+1,\label{LFeq1}\\[5pt]
        & 4z[a_{n+1}(T_{n+2}+T_n)-a_n(T_{n+1}+T_{n-1})-T_n(R_n+R_{n-1})+T_{n+1}(R_{n+1}+R_n)]=b_n,\label{LFeq12}
        \end{align}
    where $R_n=a_{n+1} +b_n^2+a_n$ and $T_n=a_n(b_n+b_{n-1})$.
\end{teo}
\begin{proof}
    Let $h_n = \langle \mathbf{u}_z, P_n^2 \rangle$. By the orthogonality of $(P_n)_{n\geq 0}$, the coefficients in \eqref{ttrrr} satisfy
\begin{equation}\label{coeffs}
    b_n h_n = \langle \mathbf{u}_z, x P_n^2 \rangle, \quad
    a_n h_{n-1} = \langle \mathbf{u}_z, x P_n P_{n-1} \rangle = h_n.
\end{equation}

We first observe that  \eqref{1eq} is equivalent to
\[
\langle \mathbf{u}_z, \psi P_n^2 \rangle = 2\, \langle \mathbf{u}_z, \phi P_n P_n' \rangle,
\]
where $\phi(x) = x$ and $\psi(x) = 4z x^4 - 1$. Using \eqref{4rr}, we compute
\begin{align*}
    \langle \mathbf{u}_z, (4z x^4 - 1) P_n^2 \rangle
    &= 4z \langle \mathbf{u}_z, x^4P_n (P_n) \rangle - h_n \\
&= 4z \left[ a_{n+2} a_{n+1} + (b_{n+1} + b_n)^2 a_{n+1} + (a_{n+1} + b_n^2 + a_n)^2 \right] h_n \\
    &\quad + 4z \left[ (b_n + b_{n-1})^2a_n + a_na_{n-1} \right] h_n - h_n.
\end{align*}

On the other hand, since
$$
P_n'(x) = n P_{n-1}(x) + \mathcal{O}(x^{n-2}),
$$
we get
\begin{align*}
    \langle \mathbf{u}_z, x P_n P_n' \rangle
    &= \left\langle \mathbf{u}_z, x P_n \left(n P_{n-1}(x) + \mathcal{O}(x^{n-2}) \right) \right\rangle = n h_n.
\end{align*}
Thus, from $\langle \mathbf{u}_z, (4z x^4 - 1) P_n^2 \rangle = 2 \langle \mathbf{u}_z, x P_n P_n' \rangle$, we deduce \eqref{LFeq1}.

\bigskip

Similarly, equation \eqref{2eq} can be rewritten as
\[
\langle \mathbf{u}_z, \psi P_n P_{n+1} \rangle = \langle \mathbf{u}_z, \phi (P_n P_{n+1}' + P_{n+1} P_n') \rangle.
\]
From  \eqref{ttrr2} and \eqref{coeffs}, we deduce
\begin{align*}
    \langle \mathbf{u}_z, (4z x^4 - 1) P_n P_{n+1} \rangle
    &= 4z \langle \mathbf{u}_z, x^2 P_n (x^2 P_{n+1}) \rangle \\
    &= 4z \Big[ (b_{n+2} + b_{n+1}) h_{n+2} + (b_{n+1} + b_n) R_{n+1} h_{n+1} \\
    &\quad + R_n T_{n+1} h_n + a_{n+1} a_n T_n h_{n-1} \Big] \\
    &= 4z \left[ a_{n+1} T_{n+2} + T_{n+1} R_{n+1} + R_n T_{n+1} + a_{n+1} T_n \right] h_n.
\end{align*}

Moreover, from \eqref{der}, \eqref{cases}, and \eqref{coeffs}, we obtain
\begin{align*}
    \langle \mathbf{u}_z, x (P_n P_{n+1}' + P_{n+1} P_n') \rangle
    &= (n+1) b_n h_n - a_n \left[ (n+1) b_n + \lambda_{n+1,n} \right] h_{n-1} \\
    &= \lambda_{n+1,n} h_n.
\end{align*}

Thus,
\begin{equation}\label{11}
    4z \left[ a_{n+1} T_{n+2} + T_{n+1} R_{n+1} + R_n T_{n+1} + a_{n+1} T_n \right] = \lambda_{n+1,n}.
\end{equation}
Shifting $n \to n-1$, we obtain
\begin{equation}\label{12}
    4z \left[ a_n T_{n+1} + T_n R_n + R_{n-1} T_n + a_n T_{n-1} \right] = \lambda_{n,n-1}.
\end{equation}
Finally, subtracting \eqref{11} from \eqref{12} and using $b_n = \lambda_{n,n-1} - \lambda_{n+1,n}$, we get \eqref{LFeq12}.

 \end{proof}

 \begin{remark}
 Notice that the above expressions were deduced by Clarke and Shizgal in \cite{Clarke} by using a different approach.
 \end{remark}
As an alternative approach, see \cite[Chapter 3]{Mourad} and \cite[Theorem 4.2]{Walter1}, let us consider $v(x)=zx^4$, which is a twice differentiable function in $(0,\infty)$.  If we define $\An_n(x)$ and $\Bn_n(x)$ by
\begin{align}
\An_{n}(x)&=\dfrac{1}{h_n}\int^{\infty}_{0}\dfrac{v^\prime(x)-v^\prime(y)}{(x-y)} P_n^2(y) e^{-zy^{4}} dy + \dfrac{ P^{2}_{n} (0)}{ h_{n} x}\label{eq:A}\\[10pt]
\Bn_{n}(x)&=\dfrac{1}{h_{n-1}}\int^{\infty}_{0}\dfrac{v^\prime(x)-v^\prime(y)}{(x-y)} P_n(y)P_{n-1}(y)e^{-zy^{4} } dy + \dfrac{ P_{n} (0) P_{n-1}(0)} {h_{n-1} x}\label{eq:B}
\end{align}
where $h_n=\prodint{\un_z,P_n^2}$, then
\begin{align}
\label{B1}\Bn_{n+1}(x) + \Bn_n(x)&= (x-b_n)\An_n(x) - v^\prime(x),\\
\label{B2}a_{n+1}\An_{n+1}(x)-a_n \An_{n-1}(x)&= 1 + (x-b_n) \left[\Bn_{n+1}(x)-\Bn_n(x)\right].
\end{align}
Taking into account \eqref{ttrr2} and
$$\dfrac{v^\prime(x)-v^\prime(y)}{(x-y)}=4z(x^2+xy+y^2),$$
we get
\begin{align*}
\An_n(x)&=4z\left(x^2+b_nx+a_{n+1}+b_n^2+a_n\right)+ \frac{ P^{2}_{n} (0)}{ h_{n} x},\\
\Bn_n(x)&=4za_n\left(x+b_n+b_{n-1}\right)+  \frac{ P_{n} (0) P_{n-1}(0)} {h_{n-1} x} .
\end{align*}
The comparison of the coefficients of \eqref{B1} and \eqref{B2} yields

\begin{itemize}
\item [(i)] $4z\left(T_{n+1}+b_n R_n+T_n  \right )=\dfrac{ P^{2}_{n} (0)}{h_{n}}$,
\item[(ii)]  $4z \left(a_{n+1}R_{n+1}-a_{n}R_{n-1}+ b_n (T_{n+1}-T_n)\right)= 1+ \dfrac{P_{n}(0) (P_{n+1}(0) - a_{n} P_{n-1}(0))} {h_{n}}.$
\end{itemize}

\begin{teo} For $n\geq 0$, the coefficients $a_n$ and $b_n$ satisfy the nonlinear difference equation
\begin{multline}\label{LFI}
a_{n+1}\left(T_{n+2}+b_{n+1} R_{n+1}+T_{n+1}  \right )\left(T_{n+1}+b_n R_n+T_n  \right )\\=\left[a_{n+1}R_{n+1}+b_n T_{n+1}+a_{n+1}a_n-\dfrac{(n+1)}{4z}\right]^2
\end{multline}
\end{teo}
\begin{proof}
Notice that $$\prodint{\un_z,(P_{n}P_{n+1})^\prime}=\prodint{\un_z,P_{n}P_{n+1}^\prime}=(n+1)h_n.$$ On the other hand, using integration by parts and \eqref{ttrr2}
$$\begin{aligned}
    \prodint{\un_z,(P_{n}P_{n+1})^\prime}&=-P_{n+1}(0)P_{n}(0)+4z\prodint{\un_z,(x^2P_{n})(xP_{n+1})}\\
    &=-P_{n+1}(0)P_{n}(0)+4z\left[h_{n+2}+b_{n+1}(b_{n+1}+b_n)h_{n+1}+a_nR_{n}h_n\right].
\end{aligned}$$
Therefore
$$(n+1)=-\dfrac{P_{n+1}(0)P_{n}(0)}{h_n}+4z\left[a_{n+2}a_{n+1}+b_{n+1}T_{n+1}+a_nR_{n}\right]$$
Notice also that
$$\dfrac{P_{n+1}^2(0)}{h_{n+1}}\dfrac{P_{n}^2(0)}{h_{n}}=\dfrac{1}{a_{n+1}}\dfrac{\left(P_n(0)P_{n+1}(0)\right)^2}{h_{n}^2}$$
and from (i) the result follows.
\end{proof}

Next, from the above nonlinear equations we analyze the asymptotic behavior of the coefficients $a_n$ and $b_n$ (in $n$) appearing in the TTRR \eqref{ttrrr}. Our aim is to determine the leading-order terms in the asymptotic expansions of these coefficients.\\

According to \eqref{eq:diffab}, we consider the ansatz
\[
a_n = n^r\,\tilde{a}_n\,z^{-1/2}, \qquad b_n = n^s\,\tilde{b}_n\,z^{-1/4},
\]
where the sequences $\tilde{a}_n$ and $\tilde{b}_n$ converge to real numbers $A$ and $B$, respectively, which must be strictly positive, such as $n \to \infty$. Thus, we adopt the asymptotic approximations
\[
a_n \sim A\,n^r\,z^{-1/2}, \qquad b_n \sim B\,n^s\,z^{-1/4} \quad \text{as } n \to \infty,
\]
where $r$ and $s$ are unknown positive exponents, and $\sim$ denotes asymptotic equivalence.  Moreover
$$R_{n}\sim z^{-1/2}\left(2An^r+Bn^{2s}\right),\qquad T_{n}\sim2z^{-3/4}AB\,n^{r+s}.$$
Substituting these expressions into \eqref{LFeq1}  and \eqref{LFI} we get $r=1/2$ and $s=1/4$. Moreover, the constants $A$ and $B$ must satisfy
\begin{align}
\label{Wq1}\left[2A^2 + 8AB^2 + (2A+B^2)^2 \right]  = \dfrac{1}{2},\\
\label{Wq2}A\left(6 AB + B^3\right)^2 = \left(3A^2 + 3AB^2 - 1/4\right)^2.
\end{align}
Since \eqref{Wq1} can be rewritten as $$3A^2+6AB^2+\dfrac{B^4}{2}=1/4,$$
then replacing by \eqref{Wq2}
$$A\left(6 AB + B^3\right)^2=\dfrac{B^2}{4}\left(6 AB + B^3\right)^2.$$
Thus, $A=B^2/4$. Replacing this in \eqref{Wq2} we get
$$\dfrac{5}{4}B^4=\pm\left(\dfrac{15}{16}B^4-\dfrac{1}{4}\right)$$
 and since $B>0$ we get  that  and $B=2\,(140)^{-1/4}$ and therefore
$A=(140)^{-1/2}$ . From the above,  We conclude that
\begin{equation}\label{asintoticaba}
a_n \sim \sqrt{\frac{n}{140z}}, \qquad b_n \sim 2\,\sqrt[4]{\frac{n}{140z}}.
\end{equation}

From here
\begin{equation}\label{comp}
xP_n(x)\sim P_{n+1}(x)+b_nP_{n}(x)+\dfrac{b_n^2}{4}P_{n-1}(x)\end{equation}
The above result can be obtained using the symmetrization process to the symmetric orthogonal polynomials associated with the Freud weight function $w(x)=|x|^{2\lambda+1}e^{-zx^8}$ which is
supported on the real line \cite[Theorem 4.6.]{Ana} (see also \cite{Clarke}).

An interesting open problem is to find the asymptotic expansions in $n$ of $a_n$ and~$b_n.$
\section{Holonomic differential equation}\label{Sec, Holo}

\begin{teo}[Structure relation]\label{structure}
The monic polynomial sequence $(P_n)_{n\geqslant 0},$ orthogonal with respect to the linear functional $\un_z,$ satisfies the differential--difference relation
    \begin{equation}\label{eq:structure}
        \begin{aligned}
        x\,P'_{n+1}(x)=&(n+1)P_{n+1}(x)\\
        &+4z[a_{n+1}(T_{n+2}+T_n)+T_{n+1}(R_{n+1}+R_n)]P_n(x)\\
        &+4z[a_{n+1}a_n(R_{n+1}+R_{n-1})+T_{n+1}T_n]P_{n-1}(x)\\
        &+4za_n[a_{n-1}T_{n+1}+a_{n+1}T_{n-1}]P_{n-2}(x)\\
        &+4za_{n+1}a_na_{n-1}a_{n-2}P_{n-3}(x).
    \end{aligned}
    \end{equation}
\end{teo}

\begin{proof}
    Since $(P_n)_{n\geqslant 0}$ is a basis of the linear space of polynomials and since the degree of $ x\,P'_{n+1}(x)$ is $n+1$, we can write
    $$
    x P'_{n+1}(x) = (n+1)P_{n+1}(x)+\sum_{k=0}^n\alpha_{n+1,k}P_k(x),
    $$
    with
    $$
\alpha_{n+1,k}h_k=\langle \un_z,x P'_{n+1} P_k\rangle, \quad 0\leqslant k \leqslant n,
$$
    and $h_k=\langle \un_z, P_k^2\rangle \ne 0$. According to the Pearson equation \eqref{pearson_eq},   we have
    \begin{align*}
\alpha_{n+1,k}h_k&=\langle \un_z,x P'_{n+1} P_k\rangle\\
    &=\langle \un_z, x(P_{n+1} P_k)'\rangle -\langle \un_z, xP_{n+1}P'_k\rangle \\
    &=\langle \un_z, P_{n+1}[(4zx^4-1)P_k-xP'_k]\rangle\\
    &=4z\langle \un_z,P_{n+1} (x^4P_k)\rangle.    \end{align*}
  From the orthogonality of $(P_{n+1})_{n\geq0}$, we deduce $\alpha_{n+1,k}=0$ for $0\leq k \leq n-4$. \end{proof}
Next, we will obtain a lowering operator for $P_n$ acting on the variable $x$.
\begin{teo}[Lowering operator]
    For $n\geqslant 0$, define
    \begin{align*}
C_n(x,z) =&4z[a_{n+1}(T_{n+2}+T_n)+T_{n+1}(R_{n+1}+R_n)]\\
&+\frac{4z}{a_n}[a_{n+1}a_n(R_{n+1}+R_{n-1})+T_{n+1}T_n](x-b_n)\\
&+\frac{4z}{a_{n-1}}[a_{n-1}T_{n+1}+a_{n+1}T_{n-1}][(x-b_{n-1})(x-b_n)-a_n]\\
&+4za_{n+1}\Big((x-b_{n-2})[(x-b_{n-1})(x-b_n)-a_n]-(x-b_n) \Big)
\end{align*}
    and
    \begin{align*}
        D_n(x,z)=&-(n+1)+\frac{4z}{a_n}[a_{n+1}a_n(R_{n+1}+R_{n-1})+T_{n+1}T_n]\\
        &+\frac{4z}{a_{n-1}}[a_{n-1}T_{n+1}+a_{n+1}T_{n-1}](x-b_{n-1})\\[5pt]
        &+4za_{n+1}[(x-b_{n-2})(x-b_{n-1})-1].
    \end{align*}
    Then, the SMOP $(P_n(_{n\geqslant 0}$ with respect to the linear functional $\un_z$ satisfies
    \begin{equation}\label{eq:lowering}
    L_nP_{n+1}(x)\,=\,P_n(x), \quad n\geqslant 0,
    \end{equation}
    where $L_n =A_n(x,z)\dfrac{d}{dx}+B_n(x,z)$ and
    \begin{equation*}
    A_n(x,z)=\frac{x}{C_n(x,z)}, \quad B_n(x,z)=\frac{D_n(x,z)}{C_n(x,z)}.
    \end{equation*}
\end{teo}

\begin{proof}
    Using the TTRR \eqref{ttrrr}, we write
    \begin{align*}
        P_{n-1}(x)=&\frac{1}{a_n}[(x-b_n)P_n(x)-P_{n+1}(x)],\\
        P_{n-2}(x)=&\frac{1}{a_na_{n-1}}\Big([(x-b_{n-1})(x-b_{n})-a_n]P_n(x)-(x-b_{n-1})P_{n+1}(x) \Big),\\
        P_{n-3}(x)=&\frac{1}{a_na_{n-1}a_{n-2}}\Big((x-b_{n-2})[(x-b_{n-1})(x-b_n)-a_n]-a_{n-1}(x-b_n) \Big)P_n(x)\\
        &-\frac{1}{a_na_{n-1}a_{n-2}}[(x-b_{n-2})(x-b_{n-1})-a_{n-1}]P_{n+1}(x).
    \end{align*}
    Substituting these expressions in \eqref{eq:structure} and rearranging, we obtain the desired result.
\end{proof}
 We note that the coefficients $\alpha_{n+1,k}$ in Theorem \ref{structure} can be written as $\alpha_{n+1,k}=4z\beta_{n+1,k}$ for $n-3\leq k\leq n,$ where $\beta_{n+1,k}$ satisfy \eqref{4rr}. This implies that $C_{n}(x,z)$ defined in the above theorem can be written as
\begin{equation*}
\begin{aligned}
C_{n}(x,z)=&4z\Big[\beta_{n+1,n}+\dfrac{(x-b_n)}{a_{n}}\beta_{n+1,n-1}+\dfrac{[(x-b_{n-1})(x-b_{n})-a_n]}{a_{n}a_{n-1}}\beta_{n+1 , n-2}\\
&+\dfrac{(x-b_{n-2})[(x-b_{n-1})((x-b_{n})-a_{n})-a_{n-1}(x-b_n)]}{a_{n}a_{n-1}a_{n-2}}\beta_{n+1,n-3}\Big].
\end{aligned}
\end{equation*}

\begin{remark}	
Recall that the TTRR \eqref{jacobi} can be written in a matrix form as $x\mathbf{P}=\mathbf{J}\mathbf{P}$. Now, define
$\mathbf{L}=4z\left(\mathbf{J}^4\right)_{-}+\mathbf{N}
$ where $(\mathbf{A})_-$ denotes the strictly lower triangular part  of matrix $\mathbf{A}$   and  $\mathbf{N}$ is the diagonal matrix $\mathbf{N}=\operatorname{diag}(0,1,2,3\cdots).$
On the other hand, taking into account  the relation between the coefficients of    \eqref{eq:structure}  and the $\beta_{n,k}$ defined  in \eqref{4rr} we can write  \begin{equation*}
x\partial_{x}\Pn=\mathbf{L}\Pn.
\end{equation*}

Notice that $$\mathbf{J}(\partial_x\Pn)=x\partial_x\Pn+\Pn$$ as well as
\begin{align*}
x\mathbf{J}(\partial_x\Pn)&=(\mathbf{J}\mathbf{L})\,\Pn,\\	
x\mathbf{J}(\partial_x\Pn)&=x(x\partial_x\Pn+\Pn)=(\mathbf{L}\mathbf{J})\,\Pn+\mathbf{J}\,\Pn.
\end{align*}
The above implies the compatibility condition of the Lax pair
$$[\mathbf{J}, \mathbf{L}] \Pn:= (\mathbf{J}\mathbf{L})\,\Pn-(\mathbf{L}\mathbf{J})\,\Pn=\mathbf{J}\,\Pn.$$
\end{remark}
\begin{coro}[Raising operator]
    For $n\geqslant 0$, let define the operator
    $$
    \widehat{L}_n=-a_{n+1} A_n(x,z)\frac{d}{dx}-[a_{n+1}B_n(x,z)-(x-b_{n+1})].
    $$
    Then, the SMOP $(P_n)_{n\geqslant 0}$  with respect to the linear functional $\un_z$ satisfies
    \begin{equation}\label{eq:raising}
    \widehat{L}_nP_{n+1}(x)\,=\,P_{n+2}(x), \quad n\geqslant 0.
    \end{equation}
\end{coro}
\begin{proof}
    Using \eqref{eq:lowering} and \eqref{ttrrr}, we get
    \begin{align*}
        xP'_{n+1}(x)=&-D_n(x,z)P_{n+1}(x)+C_n(x,z)P_n(x)\\
        =&-D_n(x,z)P_{n+1}(x)+\frac{C_n(x,z)}{a_{n+1}}[(x-b_{n+1})P_{n+1}(x)-P_{n+2}(x)]\\
        =&-\left[D_n(x,z)- \frac{C_n(x,z)}{a_{n+1}}(x-b_{n+1})\right]P_{n+1}(x)-\frac{C_n(x,z)}{a_{n+1}}P_{n+2}(x).
    \end{align*}
     Finally,  \eqref{eq:raising} follows by rearranging the terms.
\end{proof}
Using the lowering operator \eqref{eq:lowering}, we deduce a second-order holonomic differential equation that acts on the variable $x$ for $P_{n+1}(x)$.

\begin{teo}[Second-order holonomic differential equation]
For $n\geqslant 0$, let define the second-order differential operator
\begin{align*}
\mathcal{D}_{n}=&a_nA_{n}(x,z)A_{n-1}(x,z)\frac{d^2}{dx^2}\\
&+[A_n(x,z)(a_nB_{n-1}(x,z)-x+b_n)+a_nA_{n-1}(x,z)(A_n'(x,z)+B_n(x,z))]\frac{d}{dx}\\
&+[a_nB_n'(x,z)A_{n-1}(x,z)+B_n(x,z)(a_nB_{n-1}(x,z)-x+b_n)+1].
\end{align*}
Then, $\mathcal{D}_nP_{n+1}(x)=0$ for $n\geqslant 0$.
\end{teo}
\begin{proof}
    Using the lowering operator \eqref{eq:lowering} the TTRR  reads as
\begin{equation}\label{eq:ttrrL_n}
0 = \big[a_n L_{n-1}L_n - (x - b_n)L_n + 1\big]P_{n+1}(x).
\end{equation}
To proceed, we compute the action of \( L_{n-1}L_n \) on \( P_{n+1}(x) \):
\begin{align*}
    L_{n-1}L_nP_{n+1}(x)
    &= \left(A_{n-1}(x,z)\frac{d}{dx} + B_{n-1}(x,z)\right)\left(A_n(x,z)\frac{d}{dx} + B_n(x,z)\right)P_{n+1}(x) \\
    &= A_{n-1}(x,z)A_n(x,z)P_{n+1}''(x) \\
    &\quad + \big[A_{n-1}(x,z)(A_n'(x,z) + B_n(x,z)) + A_n(x,z)B_{n-1}(x,z)\big]P_{n+1}'(x) \\
    &\quad + \big[B_n'(x,z)A_{n-1}(x,z) + B_n(x,z)B_{n-1}(x,z)\big]P_{n+1}(x).
\end{align*}
This expansion provides a full expression for \( L_{n-1}L_nP_{n+1}(x) \) in terms of the polynomial \( P_{n+1}(x) \) and its derivatives. By substituting this result, along with the definition of \( L_n \), into \eqref{eq:ttrrL_n}, the desired result follows.
\end{proof}
\section{The role of the variable $z$} \label{section6}
Observe that making a change of variable in the integral \eqref{lfpd} we see that the moments of the linear functional $\un_z$ can be written as
\begin{equation}\label{momenttransf}
    \mu_n(z)=\dfrac{1}{z^{(n+1)/4}}\mu_n(1)
\end{equation}
Similarly, we see that the sequence of polynomials $
(z^{n/4} P_{n} (x z^{-1/4} ; z))_{n\geq0}
$
is an SMOP with respect to the linear functional ${\un}_{1}$. In other words, \begin{equation}\label{zerosrelation}
    z^{n/4}  P_{n} (x z^{-1/4} ; z)= P_{n} (x; 1).
\end{equation}
Indeed, a relation can be  established  between the coefficients of the corresponding TTRR when a dilation in the variable is introduced; see \cite[Chapter 1, \S 4.1]{Chi}.

In this section, we will study the moments $(\mu_n(z))_{n\geqslant 0}$ and the coefficients $a_n, b_n$ of the TTRR as functions of $z$.

The first direct result appears to take derivative in \eqref{momenttransf} with respect to the variable $z$
$$
4z\partial_z\mu_n(z)=-(n+1)\mu_n(z), \quad n\geqslant 0.
$$
If we define $\sigma_n(z)$ by
\begin{equation}\label{sigma}
P_n(x;z)=x^n-\sigma_n(z)x^{n-1}+\mathcal{O}(x^{n-2}),
\end{equation}
then  from the TTRR \eqref{ttrrr}  it follows that
$$
b_n=\sigma_{n+1}(z)-\sigma_n(z).
$$
Before studying the asymptotic behavior of the coefficients $a_n$ and $b_n$ in the TTRR, we need to study some of their properties in terms of the variable $z$.\\
Note that if in \eqref{rev} we wrote $\sigma_n(z)=-\lambda_{n,n-1},$ then
$$
nb_{n-1}+\lambda_{n,n-1}=(n-1)b_{n-1}-\sigma_{n-1}(z).
$$
With this in mind we get
\begin{teo}
    Let $h_n = \langle \un_z, P_n^2\rangle$, $a_n,b_n$ be defined in \eqref{ttrrr}, and $\sigma_n(z)$ be defined by \eqref{sigma}. We have
    $$
    4z\partial_z\sigma_n(z)=-\sigma_n(z),
    $$
    and
    \begin{equation*}
    4z\partial_z  h_n = -(2n+1)h_n.
    \end{equation*}
\end{teo}
\begin{coro}
    The coefficients of the TTRR \eqref{ttrrr} satisfy
    \begin{equation}\label{eq:diffab}
    a_n (z)= z^{-1/2} a_{n}(1),  \quad b_n (z)= z^{-1/4} b_{n}(1).
    \end{equation}
\end{coro}
\section{Zeros and their electrostatic interpretation }\label{elect_Section}
Taking into account that $\un_z$ defined in \eqref{lfpd} is a positive definite linear functional, the zeros of $P_n(x,z)$ are real, simple and located in the interior of the convex hull of the support \cite{Chi,GMM21,Gabor}.
With this in mind, let $(x_{n,k}(z))_{k=1}^{n}$ be the zeros of $P_{n}(x, z)$ in increasing order, that is,
\begin{equation}\label{zerosP_n}
P_{n}(x_{n,k}(z),z)=0
\end{equation}
with
$$x_{n,1}(z)<x_{n,2}(z)<\cdots<x_{n,n}(z).$$

From Equation \eqref{zerosrelation}, a first direct result regarding the zeros can be established (see also \cite{Ismail-Wen11}).
\begin{coro}\label{7.1}
Let $(x_{n,k}(z))_{k=1}^n$ be the zeros of $P_{n}(x,z)$. Then $ x_{n, k}(z)=  z^{-1/4} x_{n,k}(1),$ for every positive real number $z.$
\end{coro}
Due to the simple relation between the zeros of $P_{n}(x,z)$ and $P_{n}(x,1)$, hereafter we will study the properties about the zeros $(x_{n,k}(1))_{k=1}^n$.
\subsection{Asymptotic zero distribution}
To consider the asymptotic distribution of the zeros of the SMOP with respect to the linear functional \eqref{lfpd} as $n\to \infty$, we use an appropriate scaling and apply the property of regular variation described in \cite{KV99}.

To state our following result, we use the notation
$$
\lim_{n/N\to t}y_{n,N}=y
$$
to denote the property that in the doubly indexed sequence $y_{n,N}$ we have
$$\lim_{j\to \infty}y_{n_j,N_j}=y$$
whenever $n_j$ and $N_j$ are two sequences of positive integer numbers such that $N_j\to \infty$ and $n_j/N_j\to t$ as $j\to \infty$.


\begin{pro}\label{pro7.2}
Let $(P_n)_{n\geq 0}$ be the SMOP with respect to \eqref{lfpd} with $z=1$. If we assume that $n$ and $N$ tend to infinity in such a way that the ratio  $n/N\to t$, then for the sequence of scaled monic orthogonal polynomials $P_{n,N}(x)=N^{-n/4}P_n(N^{1/4}x)$
 the asymptotic zero distribution as $n\to \infty$ has density
$$
\omega(x,t)=\dfrac{4}{7\,\pi}\dfrac{ x^{-1/2}}{t^{1/8}\,c^{1/2}}\,\pFq{2}{1}\left(1/2,-7/2;-5/2;\dfrac{x}{4c\,t^{1/4}}\right),\quad c=\dfrac{1}{\sqrt[4]{140}},
$$
defined on the interval $(0, 4\,c\,t^{1/4})$.
\end{pro}
\begin{proof}
The coefficients $b_{n,N}$, $n\geq 0$,  and $a_{n,N}$, $n\geq 1$, of the TTRR associated with the scaled monic orthogonal polynomials satisfy
$$b_{n,N}=\dfrac{b_{n}}{N^{1/4}},\quad \quad \sqrt{a_{n,N}}=\dfrac{\sqrt{a_n}}{N^{1/4}}.$$
Further, from \eqref{asintoticaba} we get
$$\lim_{n/N\to t}b_{n,N}=2\,c\,t^{1/4}\quad\text{and}\quad \lim_{n/N\to t}\sqrt{a_{n,N}}=c\,t^{1/4},\quad c=\dfrac{1}{\sqrt[4]{140}}.$$
The coefficients $b_{n,N}$ and $\sqrt{a_{n,N}}$ are said to be regularly varying at infinity with index $1/4$ (cf. \cite[Section 4.5]{KV99}) If we define $\alpha(t)=0$ and $\beta(t)=4c\,t^{1/4}, $ then using \cite[Theorem 1.4]{KV99}  and \cite[eq. (1.12)]{KV99}
the asymptotic zero distribution has as density
    \begin{align*}
  \omega(x,t)&=\dfrac{1}{\pi\,t}\int_0^t\dfrac{ds}{(4c\,s^{1/4}-x)^{1/2}\,x^{1/2}}=\dfrac{1}{2\pi\,t\,x^{1/2}\,c^{1/2}}\int_0^t s^{-1/8}\left(1-\dfrac{x}{4c\,s^{1/4}}\right)^{-1/2}\,ds.
    \end{align*}
Doing the change of variable $y=s^{1/4}$ we get
     \begin{align*}
  \omega(x,t)&=\dfrac{2}{\pi\,t\,x^{1/2}\,c^{1/2}}\int_0^{t^{1/4}} y^{5/2}\left(1-\dfrac{x}{4c\,y}\right)^{-1/2}dy\\
  &=\dfrac{2}{\pi\,t\,x^{1/2}\,c^{1/2}}\int_0^{t^{1/4}} y^{5/2}\sum_{k=0}^\infty\dfrac{(1/2)_k}{k!}\left(\dfrac{x}{4c\,y}\right)^kdy\\
  &=\dfrac{2}{\pi\,t\,x^{1/2}\,c^{1/2}}\sum_{k=0}^\infty\dfrac{(1/2)_k}{k!}\left(\dfrac{x}{4c}\right)^k\int_0^{t^{1/4}} y^{5/2-k}dy\\
  &=\dfrac{4}{\pi\,t^{1/8}\,x^{1/2}\,c^{1/2}}\sum_{k=0}^\infty\dfrac{(1/2)_k}{k!}\dfrac{1}{7-2k}\left(\dfrac{x}{4c\,t^{1/4}}\right)^k\\
  &=\dfrac{4}{7\,\pi\,t^{1/8}\,x^{1/2}\,c^{1/2}}\sum_{k=0}^\infty\dfrac{(1/2)_k}{k!}\dfrac{(-7/2)_k}{(-5/2)_k}\left(\dfrac{x}{4c\,t^{1/4}}\right)^k
  \end{align*}
  and the result follows.
\end{proof}
Figure \ref{di1} illustrates the limit distribution of zeros for different values of $t$ according to Proposition \ref{pro7.2}.
\begin{figure}[ht]
\includegraphics[width=\textwidth]{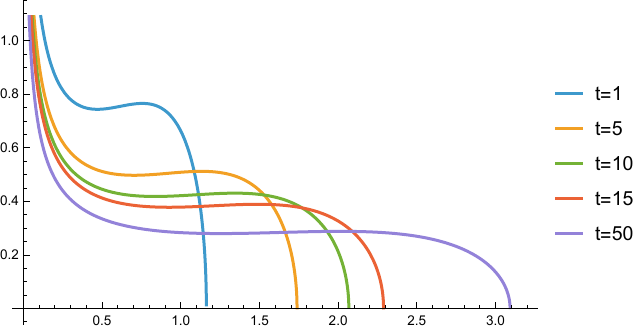}
  \caption{The limit distribution of zeros $\omega(x,t)$ for different values of $t$}
  \label{di1}
  \end{figure}
\begin{pro}
Let $(x_{n,k})_{k=1}^n$ be the zeros of $P_n(x)$ (with $z=1$) ordered as in \eqref{zerosP_n}.
For $n=2, 3,\ldots$, the largest zero, $x_{n,n}$,  satisfies
$$0<x_{n,n}<\max_{1\leq k\leq 2n-1}c_{2n}\gamma_{k},$$ where
$$\gamma_{2k+1}=-\dfrac{P_{k+1}(0)}{P_{k}(0)},\quad\quad \gamma_{2k}=-a_k\dfrac{P_{k-1}(0)}{P_{k}(0)},$$
and  $c_{2n}=4\cos^2\left(\dfrac{\pi}{2n+1}\right)+\varepsilon$, $\varepsilon>0.$
\end{pro}
\begin{proof}
Let $(S_n)_{n\geq 0}$ be the SMOP with respect to the measure $$d\mu=|x|\,e^{-x^8}dx$$
supported on $\mathbb{R}$. Since $d\mu$ is a symmetric measure, then  $(S_n)_{n\geq 0}$ satisfies a TTRR  $$xS_n(x)=S_{n+1}(x)+\gamma_{n}S_{n-1}(x).$$
Using the symmetrization process \cite[Chapter 6]{GMM21} we get that $S_{2n}(x)=P_n(x^2)$ as well as
$$\gamma_{2n+1}=-\dfrac{P_{n+1}(0)}{P_{n}(0)},\quad\quad \gamma_{2n}=-a_n\dfrac{P_{n-1}(0)}{P_{n}(0)}.$$
Let $(y_{2n,k})_{k=1}^n$ be the positive zeros   of $S_{2n}(x)$  in an increasing order, i.e.,  $y_{2n,1}<\cdots <y_{2n,n}$. Since $x_{n,k}=y_{2n,k}^2$, $k=1,\ldots, n,$ are the zeros of $P_n(x)$, then from \cite[Theorem 6.3]{Ana}, we have
$$0<\sqrt{x_{n,n}}=y_{2n,n}<\max_{1\leq k\leq 2n-1}\sqrt{c_{2n}\gamma_{k}}$$
and the result follows.
\end{proof}
\subsection{Electrostatic interpretation} Let us consider the ladder operators introduced in \cite{Chen97}, which provide a more straightforward approach to obtaining the electrostatic interpretation of the zeros of orthogonal polynomials.

Assuming that the linear functional is associated with a weight function and under suitable integrability conditions, one can derive both raising and lowering operators for orthogonal polynomials with respect to a weight function supported on the real line. Additionally, a second-order linear differential equation satisfied by these polynomials can be obtained. Notably, \cite{Chen05} also explores the case of orthogonal polynomials associated with discontinuous weights.

In the specific case of the linear functional $\un_z$, whose associated monic orthogonal polynomials are denoted by $(P_n)_{n \geqslant 0}$, the lowering and raising operators are given by
\begin{align*}
    \left(\frac{d}{dx}+{\mathcal{B}}_n(x)\right)P_n(x)=a_n\,{\mathcal{A}}_n(x)P_{n-1}(x),& \quad n\geqslant 1,\\[5pt]
    -\left(\frac{d}{dx}-{\mathcal{B}}_n(x)-v'(x)\right)P_{n-1}(x)={\mathcal{A}}_{n-1}(x)P_{n}(x),& \quad n\geqslant 1,
\end{align*}
 where $\An_n(x)$ and $\Bn_n(x)$ are given in \eqref{eq:A} and \eqref{eq:B}, respectively,  $v(x)=zx^4$ and  $h_n=\prodint{\un_z,P_n^2}$. Moreover, the combination of the raising and lowering operators leads to the second-order linear differential equation
\begin{equation}\label{eq:diffeq2}
P''_n(x)+S(x;n)P_n'(x)+Q(x;n)P_n(x)=0, \quad n\geqslant 1,
\end{equation}
where
\begin{align*}
    S(x:n)=-v'(x)-\frac{d}{dx}\ln {\An}_n(x),
\end{align*}
and
$$
Q(x;n)=-{\Bn}'_n(x)-{\Bn}_n(x)\frac{d}{dx}\ln {\An}_n(x)-{\Bn}_n(x)[v'(x)+{\Bn}_n(x)]+\frac{h_{n-1}}{h_n}{\An}_n(x){\An}_{n-1}(x).
$$
Let us denote the zeros of \(P_n(x)\) by \((x_{n,k})_{k=1}^n\) in increasing order, that is,
\[
P_n(x_{n,k}) = 0, \quad 1 \leqslant k \leqslant n, \quad \text{with} \quad x_{n,1} < x_{n,2} < \cdots < x_{n,n}.
\]

Evaluating \eqref{eq:diffeq2} at \(x = x_{n,k}\), we obtain:
\begin{equation}\label{eq:diffeqxnk}
\frac{P_n''(x_{n,k})}{P_n'(x_{n,k})} = -S_n(x_{n,k}) = 4z x_{n,k}^3 + \left.\frac{d}{dx}\ln {\An}_{n}(x)\right|_{x = x_{n,k}}, \quad n \geqslant 1.
\end{equation}

Next, observe that
\[
\ln {\An}_{n}(x) = \ln \frac{4h_n z x (x^2 + b_n x + a_{n+1} + b_n^2 + a_n) + P_n^2(0)}{h_n x}.
\]

Therefore
\begin{equation}\label{eq:derivlnA}
\frac{d}{dx}\ln {\An}_{n}(x) = \frac{3x^2 + 2b_n x + a_{n+1} + b_n^2 + a_n}{x(x^2 + b_n x + a_{n+1} + b_n^2 + a_n) + \dfrac{P_n^2(0)}{4z h_n}} - \frac{1}{x}.
\end{equation}

\begin{teo}
The zeros of \(P_n(x)\) can be interpreted as the equilibrium positions of \(n\) unit-charged particles distributed along the interval \((0, \infty)\) under the influence of the potential
\begin{equation}\label{eq:potential}
V_n(x) = z x^4 + \ln\left|x(x^2 + b_n x + a_{n+1} + b_n^2 + a_n) + \dfrac{P_n^2(0)}{4z h_n}\right|- \ln|x| .
\end{equation}
\end{teo}
\begin{proof}
    If we write $P_n(x)=(x-x_{n,1})\cdots(x-x_{n,n})$, then according to \cite[Ch. 10]{GMM21},
    $$
    \frac{P_n''(x_{n,k})}{P_n'(x_{n,k})} =\sum_{\stackrel{j=1}{j\ne k}}^n\frac{2}{x_{n,k}-x_{n,j}},
    $$
    and so, \eqref{eq:diffeqxnk} and \eqref{eq:derivlnA} yield
    $$
    \sum_{\stackrel{j=1}{j\ne k}}^n\frac{2}{x_{n,k}-x_{n,j}}-4zx_{n,k}^3-\frac{3x_{n,k}^2 + 2b_n x_{n,k} + a_{n+1} + b_n^2 + a_n}{x_{n,k}(x_{n,k}^2 + b_n x_{n,k} + a_{n+1} + b_n^2 + a_n) + \dfrac{P_n^2(0)}{4z h_n}} + \frac{1}{x_{n,k}}=0,
    $$
    or equivalently,
    $$
    \frac{\partial E_n}{\partial x_{n,k}}=0, \quad k=1,\ldots,n,
    $$
    where the total energy of the system, $E_n:=E_n(x_{n,1},\ldots,x_{n,n})$, is given by
    \begin{align*}
    E_n&=-2\sum_{1\leqslant j < k \leqslant n} \ln |x_{n,k}-x_{n,j}|\\
    &+\sum_{k=1}^n\left[zx^4_{n,k}+\ln\left|x_{n,k}(x_{n,k}^2 + b_n x_{n,k} + a_{n+1} + b_n^2 + a_n) + \dfrac{P_n^2(0)}{4z h_n}\right|-\ln|x_{n,k}|\right].
    \end{align*}
    Consequently, the external potential $V_n(x)$ is expressed as \eqref{eq:potential}.
\end{proof}

\subsection{Numerical experiments}
In this subsection, we present numerical experiments that illustrate the behavior of the zeros of the truncated Freud polynomials studied in this work. Our analysis focuses on the smallest and largest zeros of each polynomial \( P_n(x) \). In light of Proposition~\ref{7.1}, we fix the parameter \( z = 1 \). For \(n=1,\ldots,14\), the polynomials are constructed by Gram–Schmidt orthogonalization of the monomial basis, using the moments in~\eqref{moments}. Tables~\ref{tabla1} and~\ref{tabla2} report the smallest and largest zeros, respectively, of the polynomials \( P_n(x) \) for \( n = 1, \ldots, 14 \).

\vspace{10pt}
\begin{minipage}{0.48\textwidth}
\begin{equation*}
\renewcommand{\arraystretch}{1.1}\begin{array}{|c|c||c|c|}\hline
n&x_{n,1}&n&x_{n,1}    \\ \hline\hline
1&0.4889&8&0.0302\\
2&0.2363&9&0.0249\\
3&0.1372&10&0.0209\\
4&0.0901&11&0.0178\\
5&0.0640&12&0.0154\\
6&0.0480&13&0.0135\\
7&0.0375&14&0.0115\\ \hline
\end{array}
\end{equation*}
\captionof{table}{Sequence  $x_{n,1}$, for $n=1,\ldots,14.$ }
\label{tabla1}
\end{minipage}
\hfill 
\begin{minipage}{0.48\textwidth}
\begin{equation*}
\renewcommand{\arraystretch}{1.1}\begin{array}{|c|c||c|c|}\hline
n&x_{n,n}&n&x_{n,n} \\ \hline\hline
1&0.4889&8&1.6804\\
2&0.8808&9&1.7522\\
3&1.1103&10&1.8174\\
4&1.2740&11&1.8771\\
5&1.4024&12&1.9323\\
6&1.5088&13&1.9843\\
7&1.6002&14&2.0393\\
\hline
\end{array}
\end{equation*}
\captionof{table}{Sequence  $x_{n,n}$, for $n=1,\ldots,14.$ }
\label{tabla2}
\end{minipage}

It is clear, that since the linear functional is positive definite and supported on the interval \( (0, \infty) \), we have
\[
\lim_{n\to \infty} x_{n,1} = 0, \quad \lim_{n\to \infty} x_{n,n} = \infty.
\]
However, our interest is to analyze the rate of convergence.  Figure~\ref{Diagrama1} displays a monotonic trend of the smallest and largest zeros of the orthogonal polynomials for low values of \( n \).
\begin{figure}[ht]
\begin{subfigure}{.5\textwidth}
\includegraphics[width=\textwidth]{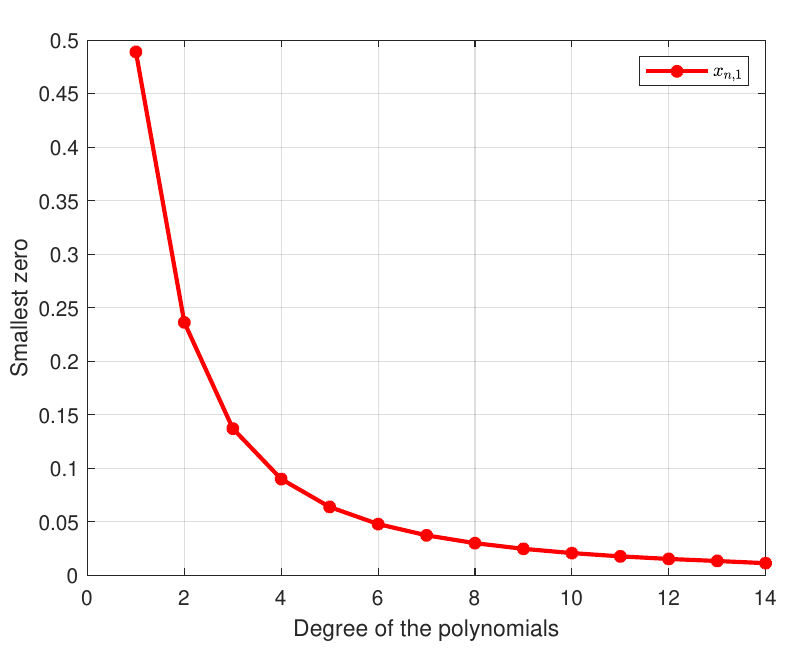}
  \caption{Smallest zeros, $n=1,\ldots, 14.$}
\end{subfigure}\begin{subfigure}{.5\textwidth}
\includegraphics[width=\textwidth]{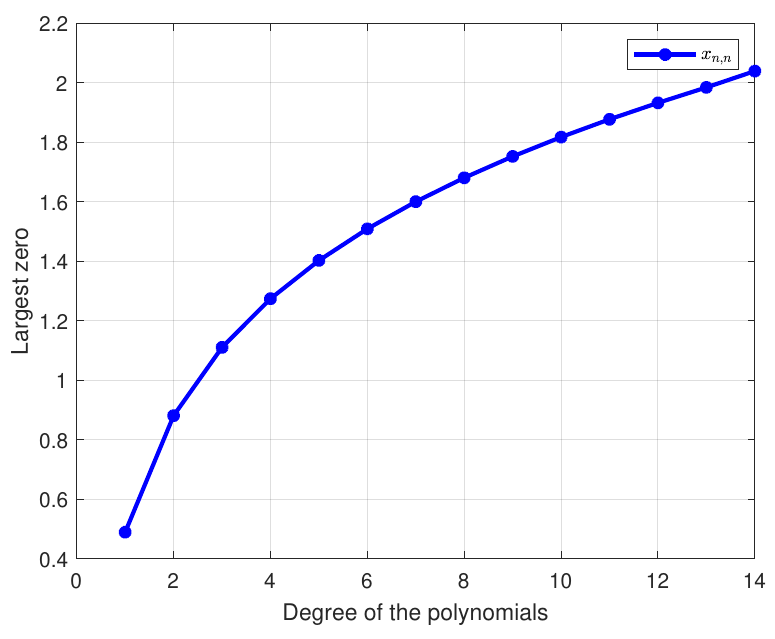}
  \caption{Largest zeros, $n=1,\ldots, 14.$}
\end{subfigure}
\caption{Behavior of smallest and largest zero for each polynomial $P_n(x)$, with $z=1$ and  $n=1,2\ldots, 14.$}
\label{Diagrama1}
\end{figure}
\begin{figure}[ht]
\begin{subfigure}{.5\textwidth}
\includegraphics[width=\textwidth]{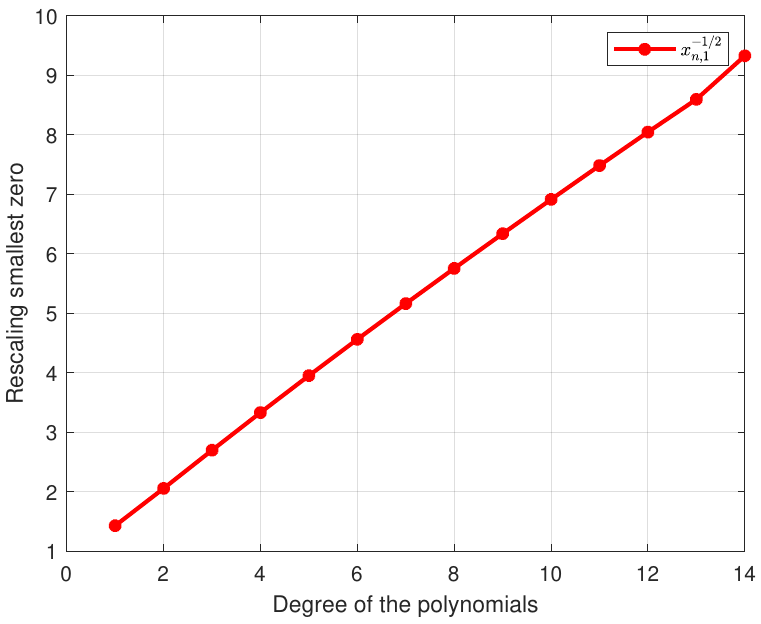}
  \caption{Smallest zeros $x_{n,1}^{-1/2}$, $n=1,\ldots, 14.$}
\end{subfigure}\begin{subfigure}{.5\textwidth}
\includegraphics[width=\textwidth]{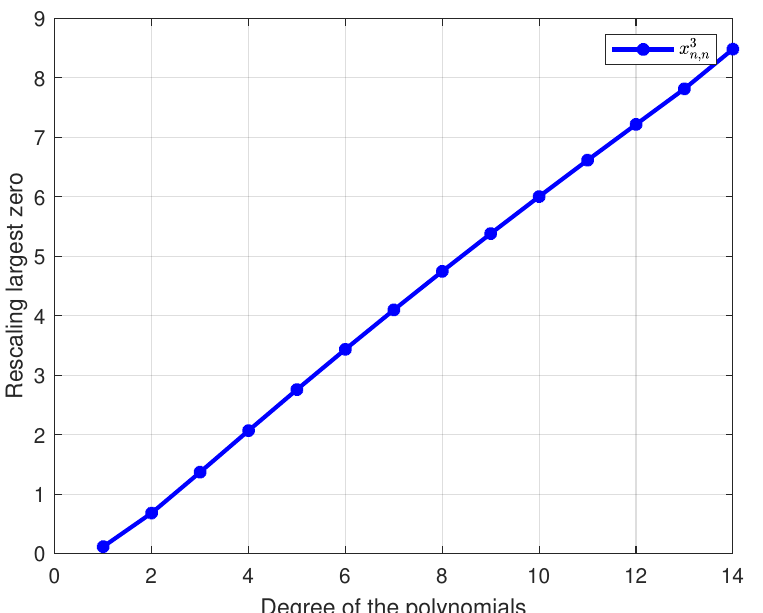}
  \caption{Largest zeros, $x_{n,n}^3,$ $n=1,\ldots, 14.$}
\end{subfigure}
\caption{Linear trend of the modified smallest and largest zero $x_{n,1}^{-1/2}$ and $x^3_{n,n}$, respectively for   $n=1,2,\ldots, 14.$}
\label{Diagrama2}
\end{figure}
Figure~\ref{Diagrama2} indicates that the extreme zeros obey the empirical asymptotics
$$
x_{n,1}\sim\mathcal{O}\!\bigl(n^{-2}\bigr),
\qquad
x_{n,n}\sim  \mathcal{O}\!\bigl(n^{1/3}\bigr).
$$
Let \(\beta := 2(1/140)^{1/4}\) and define the SMOP \((Q_n)_{n\geqslant 0}\) by the TTRR
$$
\begin{cases}
xQ_n(x)=Q_{n+1}(x)+\beta\, Q_n(x)+\dfrac{\beta^2}{4}Q_{n-1}(x), \quad n\geq0, \\[10pt]
Q_0(x)=1, \quad Q_1(x)=x-\beta.
\end{cases}
$$
 If we denote by $(\widehat{U}_n)_{n\geq 0}$ the sequence of monic Chebyshev polynomials of the second kind satisfying the TTRR
$$
\begin{cases}
x\widehat{U}_n(x)=\widehat{U}_{n+1}(x)+\dfrac{1}{4}\widehat{U}_{n-1}(x),\quad n\geq 0,\\[10pt]
\widehat{U}_0(x)=1, \quad \widehat{U}_1(x)=x,
\end{cases}
$$
then we can represent the polynomial $Q_n(x)$ as follows,
$$Q_{n}(x)=\beta^n\,\widehat{U}_n\left(\dfrac{x-\beta}{\beta}\right).$$
Their zeros written in increasing order are
$$y_{n,k}=\beta\left(\cos\left(\frac{(n-k+1)}{n+1}\pi\right)+1\right),\quad k=1,\ldots, n.$$
Clearly, $(y_{n,1})_{n\geqslant 0}$ is a bounded sequence. In fact,
$$\lim_{n\to 0}y_{n,1}=0,$$
and
$$\lim_{n\to\infty}\dfrac{y_{n,n}}{w(n)}=1,\quad \text{
where}\quad w(n)=\beta\dfrac{\pi^2}{2(n+1)^2}.$$
In Figure \ref{Fig:Chebzeros}, we present a comparison of the values of $y_{n,1}$ and $x_{n,1}$ for $n=1,\ldots,14$.
\begin{figure}[ht]
\includegraphics[width=\textwidth]{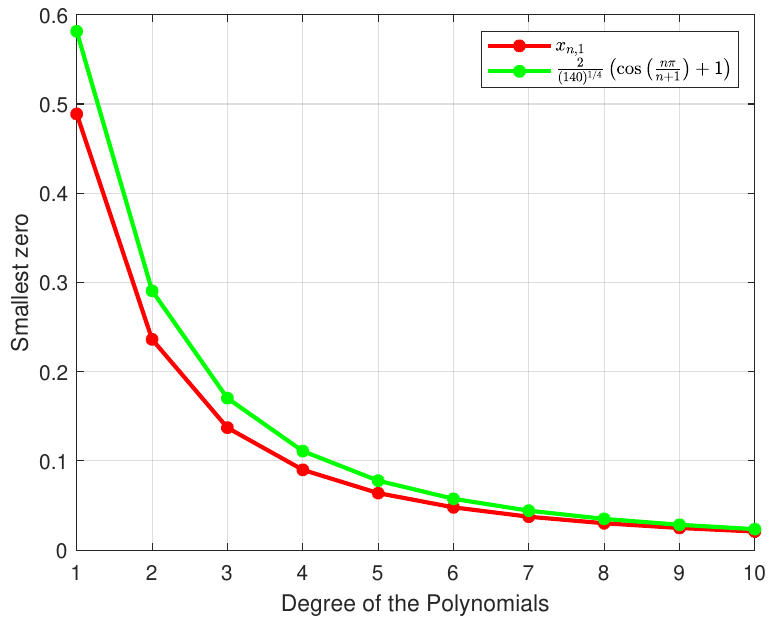}
  \caption{Comparative of $y_{n,1}=\frac{2}{\sqrt[4]{140}}\left(\cos\left(\frac{n\pi}{n+1}\right)+1\right)$ with the smallest zero $x_{n,1}$ for $n=1,\ldots, 14.$}\label{Fig:Chebzeros}
\end{figure}

Now, according to \eqref{comp}, we define the SMOP $(\widetilde{P}_n)_{n\geqslant 0}$  by means of the  TTRR
$$
\begin{cases}
x\widetilde P_n(x)= \widetilde{P}_{n+1}(x)+\sqrt[4]{n}\beta\widetilde P_{n}(x)+\sqrt{n}\dfrac{\beta^2}{4}\widetilde P_{n-1}(x), \quad n\geq0, \\[10pt]
\widetilde{P}_0(x)=1, \quad \widetilde{P}_1(x)=x.
\end{cases}
$$
Denote by $(\widetilde{x}_{n,k})_{k=1}^n$ the zeros of $\widetilde{P}_n(x)$ in increasing order.
Figure \ref{Fig:assymptzeros} shows a comparison of the values of $x_{n,n}$ and $\widetilde{x}_{n,n}$ for $n=1,\ldots,14$.
\begin{figure}[ht]
\includegraphics[width=\textwidth]{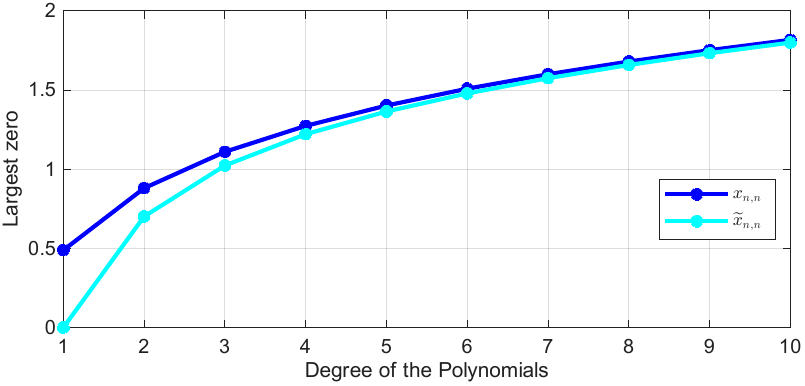}
  \caption{Comparative of the largest zeros $x_{n,n}$ and  $\widetilde x_{n,n}$ for $n=1,\ldots, 14.$}\label{Fig:assymptzeros}
\end{figure}
\section{Concluding remarks.}
In this contribution, we have analyzed some algebraic properties of SMOP associated with the so-called Druyvesteyn distribution function, which is a truncated Freud distribution function supported on the positive real semiaxis. The semiclassical character of such polynomials yields a system of nonlinear difference equations (Laguerre-Freud equations)  that the coefficients of the TTRR satisfy. Raising and lowering operators are deduced and, as a consequence, a second-order linear differential equation for such a SMOP is deduced. An electrostatic interpretation of their zeros is deduced. Finally, some illustrative numerical tests concerning the behavior of the least and greatest zeros of these polynomials are presented.\\

In a work in progress, we are dealing with similar problems for SMOP associated with the generalized Druyvesteyn distribution function $\omega(x)= x^{p} e^{-z x^{4}}, p>0, $ supported on the positive real semi-axis. According to the symmetrization process described in \cite{DGM23}, we point out that the analysis of orthogonal polynomials with respect to the symmetric linear functional associated with the generalized Freud weight function $\omega(t)=| t|^{2p+1} e^{-z t^{8}}$ supported on the real line is a particular case of one analyzed in \cite{Ana}. In particular, the behavior of coefficients on the TTRR as well as the behavior of the largest zero described in our contribution can be deduced in a natural way from the symmetrization process.
\section*{Acknowledgements}
 The work of F. Marcell\'an and M. E. Marriaga  have been supported by the research project PID2021-122154NB-I00 entitled \emph{Ortogonalidad y Aproximaci\'on con Aplicaciones en Machine Learning y Teor\'ia de la Probabilidad} funded  by MICIU/ AEI/ 10.13039/ 501100011033 and by ``ERDF A Way of making Europe''.

\end{document}